\documentclass[11pt,pdftex]{amsart}
\usepackage{txfonts,mathptmx}
\usepackage{mathrsfs} %
\usepackage[colorlinks=true]{hyperref} %
\usepackage[msc-links]{amsrefs}%

\usepackage{tikz-cd}

\numberwithin{equation}{section}

\newtheorem{thm}{Theorem}[section]
\newtheorem{lem}[thm]{Lemma}
\newtheorem{pro}[thm]{Proposition}

\theoremstyle{remark}

\newtheorem*{notation}{Notation}
\newtheorem{rem}[thm]{Remark}
\newtheorem*{ack}{Acknowledgement}

\newcommand{\Q}{\mathbb{Q}}

\newcommand{\Z}{\mathbb{Z}}

\newcommand{\Hom}{\operatorname{Hom}}

\begin{document}
\title[Homology of even Artin groups]{
The second integral homology of even Artin groups}
\author[T. Akita]{Toshiyuki Akita}
\address{Department of Mathematics, Faculty of Science, Hokkaido University,
Sapporo, 060-0810 Japan}
\email{akita@math.sci.hokudai.ac.jp}

\maketitle

\begin{abstract}
In this paper, we compute the second integral homology of even Artin groups using Hopf's formula. We then apply our results to the computation of cup products and Pontryagin products on even Artin groups, as well as to the second integral homology of even Coxeter groups.
\end{abstract}

\section{Introduction}
Artin groups are important objects in various areas of mathematics, 
including algebraic and geometric topology, combinatorial and geometric group theory, 
hyperplane arrangements, algebraic geometry, 
singularity theory, and representation theory.
For an overview of Artin groups and the $K(\pi, 1)$ conjecture---which 
remains the most important open problem concerning Artin groups---see 
the survey articles by Paris \cites{MR2497781,MR3205598,MR3207280}.
Recently, \emph{even Artin groups}---including the well-known class of  
right-angled Artin groups---have been attracting attention from several mathematicians.
For a comprehensive overview of even Artin groups, 
see the thesis of Blasco-Garc\'{\i}a \cite{garcia-thesis}.

For each Artin group, Charney and Davis \cite{MR1368655}, and independently Salvetti \cite{MR1295551}, constructed a finite CW complex known as the Salvetti complex.
The $K(\pi,1)$ conjecture is equivalent to the assertion that 
this complex serves as a classifying space for the corresponding Artin group.
Without assuming the $K(\pi,1)$ conjecture, computing the homology of Artin groups in general is difficult.
The strongest known result in this direction is the computation of the second mod~$2$ homology for \emph{all} Artin groups, achieved by the author and Liu \cite{MR3748252}. This was accomplished using Hopf's formula \cite{MR6510} together with a result of Howlett \cite{MR966298} on the second integral homology of Coxeter groups.

The focus of this paper is the computation of the second integral homology of even Artin groups.
Charney \cite{MR2058510} (see also Blasco-Garc\'{\i}a \cite{garcia-thesis})
proved that the $K(\pi,1)$ conjecture holds for even Artin groups. 
Consequently, the associated Salvetti complexes provide classifying spaces for these groups. 
This result enables us to compute their second integral homology via the homology of the Salvetti complexes, as done by Clancy and Ellis \cite{MR2672155}, who computed
the second integral homology of Salvetti complexes for \emph{all} Artin groups.

In this paper, we adopt a different approach.
We compute the second integral homology of even Artin groups by determining an explicit basis using Hopf's formula (Theorem~\ref{thm:main}).
Our method is purely algebraic and offers several advantages.
Namely, as applications of Theorem~\ref{thm:main}, 
we establish the following results for any even Artin group~$A$:  
(1) We determine the cup product map
$\smile\colon H^1(A)\otimes_\Z H^1(A)\to H^2(A)$
(Theorem~\ref{thm:cup-prod}).
(2) We express each homology class in a basis of $H_2(A)$
as a Pontryagin product of elements of $A$
(Theorem~\ref{thm:Pontryagin}).
Both results are obtained through purely algebraic arguments.

The paper is organized as follows.
In Section~\ref{sec:Artin}, we recall the definitions of Artin groups and Coxeter groups, including even Artin groups and even Coxeter groups.
In Section~\ref{sec:Hopf}, we recall Hopf's formula and  its consequence.
Theorems~\ref{thm:main}, \ref{thm:cup-prod}, and \ref{thm:Pontryagin} are proved in Sections~\ref{sec:H2-of-Artin}, \ref{sec:cup}, and \ref{sec:Pontryagin}, respectively.
In the final section, Section~\ref{sec:Coxeter}, we apply our results to even Coxeter groups.
As a result, for any even Coxeter group $W$,
we obtain an explicit basis for $H_2(W)$, 
which is an elementary abelian $2$-group, and express the basis elements as Pontryagin products of elements of $W$.

\begin{notation}
Let \( G \) be a group, and let \( g, h \in G \). For any natural number \( m \geq 2 \), define
\[
(gh)_m \coloneqq \underbrace{ghgh\cdots}_{m \text{ letters}} \in G.
\]
For example,
$(gh)_2 = gh$, $(gh)_3 = ghg$, $(gh)_4 = ghgh$, and so on.
Let \( N \) be a normal subgroup of \( G \). For \( g, h \in G \), we write
\[
g \equiv h \mod N
\]
if and only if \( gN = hN \) in \( G/N \). We denote the coset \( gN \) by 
\[g \bmod N.\]
Finally, throughout this paper, all homology and cohomology groups are taken with coefficients in~$\mathbb{Z}$,
unless stated otherwise.
\end{notation}

\section{Artin groups and Coxeter groups}\label{sec:Artin}
Let $n$ be a positive integer and define $[n]\coloneqq \{1,2,\dots,n\}$.
A \emph{Coxeter matrix} over $[n]$ is a symmetric $n\times n$ matrix
$M=(m(i,j))_{i,j\in [n]}$ such that 
\begin{itemize}
\item $m(i,i)=1$ for all $i\in [n]$,
\item $m(i,j)=m(j,i)\in \{2,3,4,\dots\}\cup\{\infty\}$ for all $i\not= j$.
\end{itemize}
Given a Coxeter matrix $M$,
the \emph{Artin group} $A_M$ associated with $M$ is
the group generated by elements $a_i$ $(i\in [n])$, subject to 
the defining relations
\begin{equation}\label{eq:rel-Artin0}
(a_{i}a_{j})_{m(i,j)}=(a_{j}a_{i})_{m(i,j)}
\end{equation}
for all $i\not=j\in [n]$ such that $m(i,j)<\infty$.
The corresponding \emph{Coxeter group} $W_M$ 
is the group generated by elements $s_i$ $(i\in [n])$, subject to 
the relations 
\begin{equation}\label{eq:Coxeter-rel}
(s_is_j)^{m(i,j)}=1\quad (i,j\in [n],\ m(i,j)<\infty).
\end{equation}
The relations \eqref{eq:Coxeter-rel} are equivalent to 
\begin{equation}\label{eq:Coxeter-rel2}
s_i^2=1\ (i\in[n]),\quad
(s_{i}s_{j})_{m(i,j)}=(s_{j}s_{i})_{m(i,j)}\
(i\not=j\in [n], m(i,j)<\infty),
\end{equation}
and hence, there exists a canonical surjection $A_M\to W_M$ defined by
$a_i\mapsto s_i$ $(i\in [n])$,  whose kernel is known as the 
\emph{pure Artin group} associated with $M$.

We say that a Coxeter matrix $M$ is \emph{even} if 
$m(i,j)\in 2\mathbb{N}\cup\{\infty\}$ for all $i\not=j\in [n]$.
The corresponding Artin group $A_M$ and Coxeter group $W_M$
are then called \emph{even}. 
In particular, if $m(i,j)\in\{2,\infty\}$ for all $i\not=j$,
the matrix and the groups are said to be \emph{right-angled}.

For an even Coxeter matrix $M$, the above relations 
\eqref{eq:rel-Artin0} and \eqref{eq:Coxeter-rel2} simplify to
\begin{equation}\label{eq:rel-Artin1}
(a_{i}a_{j})^{n(i,j)}=(a_{j}a_{i})^{n(i,j)}\quad
(i\not=j\in [n], n(i,j)<\infty)
\end{equation}
and
\begin{equation}\label{eq:rel-Coxeter1}
s_i^2=1\quad (i\in[n]),\quad
(s_{i}s_{j})^{n(i,j)}=(s_{j}s_{i})^{n(i,j)}\quad
(i\not=j\in [n], n(i,j)<\infty)
\end{equation}
respectively, where we set $n(i,j)\coloneqq m(i,j)/2$ if $m(i,j)<\infty$, and 
define $n(i,j)\coloneqq\infty$ if $m(i,j)=\infty$.

\section{Hopf's Formula}\label{sec:Hopf}

Let $G = F / R$ be a group, where $F$ is a free group and $R$ is a normal subgroup of $F$.
Then Hopf's formula \cite{MR6510} (see also \cite{MR672956}) gives:
\[
H_2(G) \cong \frac{R \cap [F, F]}{[F, R]}.
\]
\begin{lem}\label{lem:Hopf-representative}
Let $G=F/R$ be a group, where $F$ is a free group and $R$ is the normal closure of
finitely many elements $r_1, r_2, \dots, r_m \in F$.
Then any second homology class in $H_2(G)$ is represented by an element in
$R\cap [F,F]$ of the form
\[
\prod_{i=1}^{m} r_i^{\ell(i)} \in R\cap [F, F]
\]
for some integers $\ell(i) \in \mathbb{Z}$.
\end{lem}
\begin{proof}
The quotient $R / [F, R]$ is an abelian group since $[R,R]\subset [F,R]$.
Let $f \in F$ and $r \in R$.
Then, since $[f,r]=frf^{-1}r^{-1}\in [F,R]$, we have
\[
f r f^{-1} \equiv r \mod [F, R].
\]
This implies that conjugation by elements of \( F \) acts trivially modulo \( [F, R] \),  
and hence any element of \( R / [F, R] \) is represented by a product of powers of the 
normal generators \( r_i \):
\[
\prod_{i=1}^{m} r_i^{\ell(i)} \in R.
\]
Together with Hopf's formula,
we conclude that any element of \( H_2(G) \) has a representative of the desired form in \( R \cap [F, F] \).
\end{proof}

\section{The first and second homology of even Artin groups}
\label{sec:H2-of-Artin}
In this section, we determine the second integral homology of all even Artin groups, as stated in Theorem~\ref{thm:main}.
Let $M$ be an even Coxeter matrix on $[n]$, and let $A_{M}$ be the associated even
Artin group.
The group $A_{M}$ is generated by elements $a_{i}$ $(i\in [n])$, subject to the relations 
\[
(a_{i}a_{j})^{n(i,j)}=(a_{j}a_{i})^{n(i,j)}\quad
(i\not=j\in [n], n(i,j)<\infty).
\]
From the presentation of $A_M$, it is easy to prove the following:
\begin{pro}\label{pro:1st-homology-basis}
The first homology group $H_{1}(A_{M})\cong A_{M}/[A_{M},A_{M}]$ is 
a free abelian group with a basis
$\{a_{i}\bmod [A_{M}, A_{M}] \mid i \in [n]\}$.
\end{pro}
\noindent
Define 
\[
\mathcal{B}\coloneqq
\{(i,j)\mid 1\leq i<j\leq n,n(i,j)<\infty\}.
\]
We write $A_{M}=F/R$, 
where $F$ is the free group on $\{a_{i}\mid i\in [n]\}$, and $R$ is the
normal closure of the set
\[
\{(a_{i}a_{j})^{n(i,j)}(a_{j}a_{i})^{-n(i,j)}\mid (i,j)\in\mathcal{B}\}\subset F.
\]
Since each relator $(a_{i}a_{j})^{n(i,j)}(a_{j}a_{i})^{-n(i,j)}$ lies in $[F,F]$,
it follows that $R\subset [F,F]$.
Therefore, by Hopf's formula, we have
\[H_{2}(A_{M})\cong R/[F,R].\]
Now we present and prove the main result of this paper:
\begin{thm} \label{thm:main}
The second integral homology group $H_{2}(A_{M})\cong R/[F,R]$ 
is a free abelian group
whose basis is given by
\[\{[a_{i},a_{j}]^{n(i,j)}\bmod [F,R]\mid (i,j)\in\mathcal{B}\}.\]
\end{thm}
\noindent
We begin the proof of the theorem by establishing the following lemma, which is the key to the argument.
\begin{lem} Let $G$ be a group, and let $a,b\in G$.
For every integer $n\geq 1$, 	
there exists an element $w_n\in [G,[G,G]]$ such that 
\[(ab)^n(ba)^{-n}=w_n[a,b]^n.\]
Consequently 
\[(ab)^n(ba)^{-n}\equiv [a,b]^n\mod [G,[G,G]].\]
\end{lem}
\begin{proof}
We prove the lemma by induction on $n$.
The case $n=1$ is immediate. %
Assume the assertion holds for some $n\geq 1$, i.e., there exists $w_n\in [G,[G,G]]$
such that 
\[(ab)^n(ba)^{-n}=w_n[a,b]^n.\]
Using the identity
\begin{equation}\label{eq:commutator}
gh=[g,h]hg\quad (g,h\in G),
\end{equation}
we compute:
\begin{align*}
(ab)^{n+1}(ba)^{-(n+1)}&=(ab)(ab)^n(ba)^{-n}(ba)^{-1} \\
&=(ab)(w_n[a,b]^n)(ba)^{-1}\\ %
&=[ab,w_n [a,b]^n](w_n[a,b]^n) (ab)(ba)^{-1}\quad
 (\text{by}\ \eqref{eq:commutator})\\
&=[ab,w_n [a,b]^n] w_n[a,b]^{n+1}.
\end{align*}
Since $w_n\in [G,[G,G]]\subset [G,G]$,
it follows that $w_n [a,b]^n\in [G,G]$, and hence
\[[ab,w_n [a,b]^n]\in [G,[G,G]].\]
Define $w_{n+1}\coloneqq [ab,w_n [a,b]^n] w_n\in [G,[G,G]]$.
This completes the inductive step and proves the lemma.
\end{proof}
\begin{proof}[Proof of Theorem \ref{thm:main}]
Since $[R,R]\subset F$ and hence $[[R,R],R]\subset [F,R]$, we see that
\[
(a_{i}a_{j})^{n(i,j)}(a_{j}a_{i})^{-n(i,j)}
\equiv [a_{i},a_{j}]^{n(i,j)}\mod [F,R].
\]
By Lemma \ref{lem:Hopf-representative}, 
any element of $H_{2}(A_{M})\cong R/[F,R]$ can be represented as
\[
\prod_{(i,j)\in\mathcal{B}}([a_{i},a_{j}]^{n(i,j)})^{\ell(i,j)}\in R
\]
for some integer $\ell(i,j)\in\Z$.
Therefore,
$H_{2}(A_{M})\cong R/[F,R]$ is
generated by 
\[\{[a_{i},a_{j}]^{n(i,j)}\bmod [F,R]\mid (i,j)\in\mathcal{B}\}.\]
Now consider the composition of homomorphisms
\[
\frac{R}{[F,R]}\hookrightarrow\frac{[F,F]}{[F,R]}\twoheadrightarrow
\frac{[F,F]}{[F,[F,F]]}
\]
induced by the inclusion $R\hookrightarrow [F,F]$.
By a well-known result of Hall \cites{MR0038336,MR0103215}, 
the group $[F,F]/[F,[F,F]]$ is a free abelian group
with basis 
\[\{[a_{i},a_{j}]\bmod [F,[F,F]]\mid 1\leq i<j\leq n\}.\]
Thus the image of 
\begin{equation}\label{eq:basis}
\{[a_{i},a_{j}]^{n(i,j)}\bmod [F,R]\mid (i,j)\in\mathcal{B}\}
\end{equation}
in $[F,F]/[F,[F,F]]$ is linearly independent, 
and hence so are their preimages \eqref{eq:basis}.
This proves the theorem.
\end{proof}

\section{Cup products}\label{sec:cup}
In this section, we compute the cup product map
\[
\smile\colon 
H^{1}(A_{M}) \otimes_\Z H^{1}(A_{M}) \to H^{2}(A_{M})
\]
of even Artin groups $A_M$. 
To this end, we use the following result:
Let $G=F/R$ be a group where $F$ is a free group and $R$ is a normal subgroup of $F$, and let 
$h\colon H^2(G)\to\Hom(H_2(G),\Z)$ be the natural homomorphism.
Let $\phi,\psi\in H^1(G)=\Hom(H_1(G),\Z)$ be two cohomology classes,
which we regard as homomorphisms $G\to\Z$.
By Hopf's formula, an element $\xi\in H^2(G)=[F,F]\cap R/[F,R]$
is represented by an element $\prod_i[g_i,h_i]\in [F,F]\cap R$.
The evaluation of $h(\phi\smile\psi)$ on $\xi$ can be
described by the follwing theorem, whch is 
essencially due to Hopf \cite{MR6510}. 
The alternative proof can be found in Gadgil-Kachari \cite{MR2582056}.
\begin{thm}\label{thm:Hopf-cup}
For the class $\xi\in H_2(G)$ corresponding to 
$\prod_i[g_i,h_i]\in [F,F]\cap R$,
\[
h(\phi\smile\psi)(\xi)=\sum_i
(\phi(g_i)\psi(h_i)-\psi(g_i)\phi(h_i)).
\]
\end{thm}
Now let $M$ be an even Coxeter matrix on the index set $[n]$, 
and let $A_{M}=F/R$ be the associated even Artin group,
where $F$ is the free group and $R$ is the normal subgroup
as in the previous section.
Since $H_1(A_M)$ and $H_2(A_M)$ are free abelian,
the universal coefficient theorem implies that
$H^1(A_M)$ and $H^2(A_M)$ are also  free abelian.
Define
\begin{align*}
&\alpha_i\coloneqq a_i\bmod [A_M,A_M]\in H_1(A_M)\quad (i\in [n]),\\
&\alpha_{ij}\coloneqq [a_i,a_j]^{n(i,j)}\bmod [F,R]\in H_2(A_M)
\quad ((i,j)\in\mathcal{B}),
\end{align*}
so that  $\{\alpha_i\mid i\in [n]\}$ is a basis of $H_1(A_M)$ by Proposition \ref{pro:1st-homology-basis}, 
and
$\{\alpha_{ij}\mid (i,j)\in\mathcal{B}\}$ is a basis of $H_2(A_M)$ by Theorem \ref{thm:main}.
Let
\[\{\beta_i\mid i\in [n]\} \subset H^1(A_M)\quad \text{and}\quad
\{\beta_{ij}\mid (i,j)\in\mathcal{B}\}\subset H^2(A_M)\]
denote the dual bases %
 corresponding to $\{\alpha_i\}$ and $\{\alpha_{ij}\}$, respectively.
By applying Theorem \ref{thm:Hopf-cup}, we prove the following.
\begin{thm}\label{thm:cup-prod}
Under these notations, $\beta_i\smile\beta_i=0$ for all $i$, and,
for $i<j$, we have:
\[\beta_i\smile\beta_j=\begin{cases}
n(i,j)\beta_{ij} & (n(i,j)<\infty), \\ 0 & (n(i,j)=\infty).\end{cases}
\]
\end{thm}
\begin{proof}
Since $H^2(A_M)$ is a free abelian group,
it is obvious by degree reasons that $\beta_i\smile\beta_i=0$. 
Let $h\colon H^2(A_M)\to\Hom(H_2(A_M),\Z)$ be the natural 
isomorphism.
To prove the theorem,
it is enough to evaluate $h(\beta_i\smile\beta_j)(\alpha_{pq})$
for $i<j$ and $(p,q)\in\mathcal{B}$.
According to Theorem \ref{thm:Hopf-cup}, we have
\begin{align*}
h(\beta_i\smile\beta_j)(\alpha_{pq})
&=n(i,j)\{\beta_i(a_p)\beta_j(a_q)-\beta_j(a_p)\beta_i(a_q)\}\\
&=n(i,j)(\delta_{ip}\delta_{jq}-\delta_{jp}\delta_{iq})
\end{align*}
beccause $\beta_i(a_j)=\beta_i(\alpha_j)=\delta_{ij}$ (Kronecker's delta)
by the definition of $\beta_i$.
Since $i<j$ and $p<q$, the case $(i,j)=(q,p)$ cannot occur.
We conclude that
\[
h(\beta_i\smile\beta_j)(\alpha_{pq})
=\begin{cases} n(i,j) & (i,j)=(p,q), \\ 0 &(\text{otherwise}),
\end{cases}
\]
from which the theorem follows immediately.
\end{proof}
\begin{rem}
It follows from Theorem \ref{thm:cup-prod} that the cup product map
\[\smile\colon H^1(A_M;\Q)\otimes_\Q H^1(A_M;\Q)\to H^2(A_M;\Q)\]
is surjective over the rationals.
\end{rem}
\begin{rem}
The integral cohomology rings of right-angled Artin groups were
determined in  Charney-Davis \cite{MR1368655}.
\end{rem}

\section{Pontryagin products}\label{sec:Pontryagin}
We first recall relevant facts concerning
Pontryagin products on abelian groups.
See \cite{MR672956} for details.
Let $G$ be an abelian group.
Then the multiplication map $\mu\colon G\times G\to G$ defined by
$(g,h)\mapsto gh$ is a group homomorphism.
The \emph{Pontryagin product} on $H_*(G)$ is
defined as the composite
\[
H_*(G)\otimes_\Z H_*(G)\xrightarrow{\times}
H_*(G\times G)\xrightarrow{\mu_*} H_*(G),
\]
where $\times$ is the cross product.
Then $H_*(G)$ is a strictly anti-commutative ring with
respect to Pontryagin products.
Identifying $H_1(G)$ with $G$, we obtain a ring homomorphism
\[
\psi\colon\wedge^\ast_\Z G \to H_*(G),
\]
induced by Pontryagin products.
The homomorphism $\psi$ is injective,
and it is an isomorphism in degrees $0$, $1$, and $2$.
We now focus on the degree 2 part.
Since $\psi\colon\wedge^2_\Z G\to H_2(G)$ is an isomorphism,
we may denote $g\wedge h\in H_2(G)$ instead of $\psi(g\wedge h)\in H_2(G)$
for simplicity.
In terms of the normalized bar resolution, we have
\begin{equation}\label{eq:Pont-bar}
g\wedge h=[g|h]-[h|g]\in H_2(G)\quad (g,h\in G)
\end{equation}
(see \cite{MR672956}*{\S V.6, Exercise 4}).
If $G=\Z\times\Z$, then $(1,0)\wedge (0,1)\in H_2(\Z\times\Z)\cong\Z$
is one of two generators.
Indeed, it represents the fundamental class of the torus
$T^{2}=K(\Z\times\Z,1)$, whose orientation is suitably chosen.

Now let $G$ be a group and $g,h\in G$ be commuting elements.
The Pontryagin product $\langle g,h\rangle\in H_2(G)$ was introduced in
Edmonds-Livingston \cite{MR0769296}.
Our definition of the Pontryagin product is equivalent to, but
slightly different from theirs.
Since $[g,h]=1$, there is a homomorphism 
$\phi\colon\Z\times\Z\to G$ given by
$\phi(1,0)=g$ and $\phi(0,1)=h$.
Then the \emph{Pontryagin product} of $g,h\in G$ is defined by
\begin{equation}\label{eq:Pont-nonabel}
\langle g,h\rangle\coloneqq \phi_*((1,0)\wedge (0,1))\in H_2(G).
\end{equation}
It follows from \eqref{eq:Pont-bar} that
\begin{equation}\label{eq:Pont-bar-nonabel} 
\langle g,h\rangle=[g|h]-[h|g]\in H_2(G).
\end{equation}
The following proposition can be derived  from
\eqref{eq:Pont-bar-nonabel}: 
\begin{pro}[Edmonds-Livingston \cite{MR0769296}]\label{pro:properties-Pont}
For a group $G$, one has
\begin{enumerate}
\item $\langle g,g\rangle=\langle g,1\rangle=\langle 1,g\rangle=0$,
\item $\langle g,h\rangle=-\langle h,g\rangle$ if $[g,h]=1$,
\item $\langle g,hk\rangle=\langle g,h\rangle+\langle g,k\rangle$
if $[g,h]=[g,k]=1$, \label{eq:boundary}
\item $\langle kgk^{-1},khk^{-1}\rangle=\langle g,h\rangle$ for all $k\in G$
 if $[g,h]=1$. 
\end{enumerate}
\end{pro}
\begin{proof}
We give a proof of \eqref{eq:boundary}, since 
it was omitted in \cite{MR0769296} and the remaining cases are straightforward.
Let 
\[c\coloneqq [g|h|k]-[h|g|k]+[h|k|g]\in C_3(G)\]
be the normalized $3$-chain of $G$.
A direct computation shows that
\[
\partial c=\langle g,hk\rangle-\langle g,h\rangle-\langle g,k\rangle,
\]
which proves \eqref{eq:boundary}.
\end{proof}

Pontryagin products on the symmetric groups $\mathfrak{S}_n$ $(n \geq 4)$ 
and the alternating groups $A_n$ $(n=6,7,8)$ were conisidered in Edmonds-Livingston~\cite{MR0769296}.  
Pontryagin products on non-abelian groups have been applied to knot and link theory \cites{MR0769296,MR0942427,MR1775387,MR2341312}.  
In this section, we compute Pontryagin products on even Artin groups.  
The following proposition plays a crucial role in our approach.

\begin{pro}\label{pro:Pont-Hopf}
Let $G=F/R$ where $F$ is a free group and $R$ is a normal subgroup of $F$,
and $s\colon G\to F$ be
a set-theoretic section. For commuting elements 
$g,h\in G$, the Pontryagin product $\langle g,h\rangle\in H_{2}(G)$
is represented by
\[
[s(g),s(h)]\bmod [F,R]\in R\cap [F,F]/[F,R].
\]
\end{pro}
\begin{proof}
In  \cite{MR672956}*{\S II.5, Exercise 4(c)}, 
an explicit formula is given for the isomorphism
$\phi\colon R\cap [F,F]/[F,R]\to H_2(G)$.
Specifically, let $\alpha_1,\dots,\alpha_g$, $\beta_1,\dots,\beta_g$ be elements of $F$
such that %
$r=\prod_{i=1}^g [\alpha_i,\beta_i]\in R$.
Then $\phi(r\bmod [F,R])$ is represented by %
\[
\sum_{i=1}^g\{[I_{i-1}|a_i]+[I_{i-1}a_i|b_i]
-[I_{i-1}a_ib_ia_i^{-1}|a_i]-[I_i|b_i]\},
\]
where $a_i,b_i$ denote the images of $\alpha_i,\beta_i$ in $G$,
and $I_i\coloneqq [a_1,b_1]\cdots [a_i,b_i]$.
Now let $g,h\in G$ be commuting elements, then $[s(g),s(h)]\in R$, and
\[\phi([s(g),s(h)]\bmod [F,R])\]
is represented by %
\[
[1|g]+[g|h]-[ghg^{-1}|g]-[[g,h]|h]=[g|h]-[h|g].
\]
Since $\langle g,h\rangle=[g|h]-[h|g]$, the proposition follows.
\end{proof}

\begin{thm}\label{thm:Pontryagin}
Let $M$ be an even Coxeter matrix and  $A_M=F/R$ be the
associated even Artin group.
If $(i,j)\in\mathcal{B}$ then $[a_i, (a_ja_i)_{2n(i,j)-1}]=1$, and we have
\[
\langle a_i, (a_ja_i)_{2n(i,j)-1}\rangle=\alpha_{ij}\in H_2(A_M),
\]
where $\{\alpha_{ij}\mid (i,j)\in\mathcal{B}\}$ is a basis of 
$H_2(A_M)$ as in the previous section.
Consequently, every homology class of $H_2(A_M)$ can be expressed as
a linear combination of Pontryagin products.
\end{thm}
\begin{proof}
The first assertion follows from
\begin{equation}\label{eq:comm-rel}
1=(a_ia_j)^{n(i,j)}(a_ja_i)^{-n(i,j)}=[a_i, (a_ja_i)_{2n(i,j)-1}]\in A_M.
\end{equation}
By Proposition \ref{pro:Pont-Hopf}, the Pontryagin product
$\langle a_i, (a_ja_i)_{2n(i,j)-1}\rangle$ is represented by
\[
[a_i, (a_ja_i)_{2n(i,j)-1}]\bmod [F,R]\in R/[F,R].
\]
Since $\alpha_{ij}\coloneqq [a_i,a_j]^{n(i,j)}\bmod [F,R]$ and
\begin{align*}
[a_i, (a_ja_i)_{2n(i,j)-1}]=(a_ia_j)^{n(i,j)}(a_ja_i)^{-n(i,j)}
\equiv [a_i,a_j]^{n(i,j)}\mod [F,R],
\end{align*}
the second assertion follows immediately.
\end{proof}

\begin{rem}
Second homology classes of groups are not always represented as
a linear combination of Pontryagin products.
Topologically, the Pontryagin product $\langle g,h\rangle\in H_{2}(G)$ is, 
up to sign, the image of
the fundamental class $[T^{2}]\in H_{2}(T^{2})$ under the map $(B\phi)_{*}$, where
$B\phi\colon T^{2}\to K(G,1)$ is a continuous map that induces the homomorphism $\phi$ 
on fundamental groups. Such a map $B\phi$ always exists and unique up to homotopy.

Now let $\Sigma_{g}$ be a closed orientable surface of genus $g\geq 2$,
so that $\Sigma_{g}$ is the Eilenberg-MacLane space $K(\pi_1(\Sigma_{g}),1)$.
It can be shown that, for any continuous map 
$f\colon T^{2}\to\Sigma_{g}$, the induced homomorphism
$f_{*}\colon H_{2}(T_{2})\to H_{2}(\Sigma_{g})$ is trivial.
It follows that no nontrivial homology class in $H_2(\pi_1(\Sigma_g))=H_2(\Sigma_g)\cong\Z$
can be represented by a Pontryagin product.
\end{rem}

\section{The second homology of even Coxeter groups}\label{sec:Coxeter}
In this section, we investigate the second integral homology of even Coxeter groups. While Howlett \cite{MR966298} determined the second integral homology for \emph{all} Coxeter groups---showing it to be an elementary abelian 2-group whose rank is computable from the group presentation---we employ a different approach to determine a basis of it in the case of even Coxeter groups.

Let $M$ be an even Coxeter matrix on the index set $[n]$, and let $A_M$ and $W_M$ denote the corresponding even Artin group and even Coxeter group, respectively. Write $A_M = F/R$ as in Section~\ref{sec:Artin}, and similarly let $W_M = F'/R'$, where $F'$ is the free group generated by $\{s_i \mid i \in [n]\}$, and $R'$ is the normal closure of the set
\[
\{s_i^2 \mid i \in [n]\} \cup \{(s_i s_j)^{n(i,j)} (s_j s_i)^{-n(i,j)} \mid (i,j) \in \mathcal{B}\} \subset F',
\]
arising from the relations in \eqref{eq:rel-Coxeter1}.
Define a homomorphism $\hat{\rho} \colon F \to F'$ by $a_i \mapsto s_i$, which lifts the canonical surjection $\rho \colon A_M \to W_M$. This induces the following commutative diagram:
\[
\begin{tikzcd}
H_2(A_M) \arrow[d, "\rho_*"'] \arrow[r, "\cong"] & R/[F,R] \arrow[d, "\hat{\rho}_*"] \\
H_2(W_M) \arrow[r,"\cong"]                           & R'\cap [F',F']/[F',R']               
\end{tikzcd}
\]
where $\hat{\rho}_*$ is the homomorphism induced by $\hat{\rho}$ (see \cite[Exercise II.6.3(b)]{MR672956}).
By Theorem~\ref{thm:main}, the group $R/[F,R]$ is free abelian with basis $\{\alpha_{ij} \mid (i,j) \in \mathcal{B}\}$, where $\alpha_{ij} = [a_i, a_j]^{n(i,j)} \bmod [F,R]$. Applying $\hat{\rho}_*$ to these generators yields the set
\begin{equation} \label{eq:even-Coxeter-closure}
\{[s_i, s_j]^{n(i,j)} \bmod [F', R'] \mid (i,j) \in \mathcal{B}\} \subset \frac{R' \cap [F', F']}{[F', R']}.
\end{equation}
Moreover, the author and Liu \cite{MR3748252} proved that the induced map
$
\rho_* \colon H_2(A_M) \to H_2(W_M)
$
is surjective, and becomes an isomorphism after tensoring with $\mathbb{Z}/2$:
\[
H_2(A_M) \otimes_{\mathbb{Z}} \mathbb{Z}/2 \to H_2(W_M).
\]
Note that $H_2(W_M) = H_2(W_M) \otimes_{\mathbb{Z}} \mathbb{Z}/2$
since $H_2(W_M)$ is an elementary abelian $2$-group.
Hence, the set in \eqref{eq:even-Coxeter-closure} forms a basis for $H_2(W_M) \cong R' \cap [F', F'] / [F', R']$.
\begin{thm}\label{thm:Coxeter-homology}
Let $M$ be an even Coxeter matrix, and let $W_M = F'/R'$ be the associated even Coxeter group. Then $H_2(W_M)$ is an elementary abelian $2$-group with basis
\[
\{[s_i, s_j]^{n(i,j)} \bmod [F', R'] \mid (i,j) \in \mathcal{B} \}.
\]
\end{thm}
Finally, since Pontryagin products are natural with respect to group homomorphisms, Theorem~\ref{thm:Pontryagin} yields the following result concerning 
Pontryagin products on $W_M$:
\begin{thm}
With the same notation as in the previous theorem, for each 
$(i,j) \in \mathcal{B}$, we have
$
[s_i, (s_j s_i)_{2n(i,j)-1}] = 1,
$
and
\[
\langle s_i, (s_j s_i)_{2n(i,j)-1} \rangle = [s_i, s_j]^{n(i,j)} \bmod [F', R'] \in H_2(W_M).
\]
\end{thm}

\begin{rem}
To be precise, Theorem~\ref{thm:Coxeter-homology} is \emph{not} entirely independent of Howlett's result \cite{MR966298}, for two reasons.  
First, we relied on the fact that $H_2(W_M)$ is an elementary abelian $2$-group.  
Second, the result of the author and Liu mentioned above itself depends on a result of Howlett.
\end{rem}

\begin{ack}
The author thanks Takao Satoh for bringing to his attention a result of Hall \cites{MR0038336, MR0103215}. 
He also thanks Sota Takase for explaining the proof of Proposition~\ref{pro:properties-Pont}~(3),
and to the anonymous referee for valuable suggestions that improved the clarity and readability of this paper.
This work was partially supported by JSPS KAKENHI Grant Numbers 20K03600 and 24K06727.
\end{ack}

\end{document}